\documentclass{article}
\usepackage[utf8]{inputenc}
\usepackage{graphicx,amsmath,amsthm,amssymb,hyperref}

\newtheorem{theorem}{Theorem}
\newtheorem{observation}[theorem]{Observation}
\newtheorem{lemma}[theorem]{Lemma}
\newtheorem{corollary}[theorem]{Corollary}
\newtheorem{problem}[theorem]{Problem}

\newcommand{\GG}{\mathcal{G}}
\newcommand{\CC}{\mathcal{C}}

\usepackage{xcolor}

\title{Characterization of sparse monotone graph classes with bounded domination-to-2-independence ratio}
\author{Marthe Bonamy\thanks{CNRS, LaBRI, Université de Bordeaux, France. \url{marthe.bonamy@u-bordeaux.fr} Supported by the ANR project ENEDISC ANR-24-CE48-7768-01.}\and Zdeněk Dvořák\thanks{Charles University, Prague, Czech Republic.
	\url{rakdver@iuuk.mff.cuni.cz}
	Supported by ERC-CZ project LL2328 (Beyond the Four Color Theorem) of the Ministry of Education of Czech Republic.}
\and Lukas Michel\thanks{Mathematical Institute, University of Oxford, United Kingdom. \url{lukas.michel@maths.ox.ac.uk}}\and David Mikšaník\thanks{Charles University, Prague, Czech Republic.
	\url{miksanik@iuuk.mff.cuni.cz}
	Supported by project 8J24FR005 (Tight bounds for fractional colouring) of the Ministry of Education of Czech Republic.}
}
\date{December 2025}

\begin{document}

\maketitle

\begin{abstract}
We give an exact characterization of monotone graph classes $\GG$ with bounded average degree that satisfy the following
property: The domination number of every graph from $\GG$ is bounded by a linear function of its 2-independence number.
\end{abstract}

\section{Introduction}

In this paper, we consider the relationship between two fundamental graph parameters, the domination number and the $2$-independence
number.
The \emph{domination number} $\gamma(G)$ of a graph $G$ is the minimum size of a \emph{dominating
set} in~$G$, i.e., a set $D\subseteq V(G)$ such that every vertex of $G$ belongs to $D$ or has a neighbor in $D$.
The \emph{$2$-independence number} $\alpha_2(G)$ of $G$ is the maximum size of a \emph{$2$-independent
set} in $G$, i.e., a set $A\subseteq V(G)$ such that the distance between any two distinct vertices of $A$ in $G$
is greater than two.

A 2-independent set of size $k$ cannot be dominated by less than $k$
vertices of $G$, and thus we clearly have $\alpha_2(G)\le \gamma(G)$.
On the other hand, the domination number of a graph is in general not bounded by
any function of its 2-independence number, even in sparse graphs.  Indeed, for every integer $k\ge 3$,
there exists a bipartite $3$-degenerate graph with $k(k+1)/2+1$ vertices, $2$-independence number $2$,
and domination number at least $k/2$, see~\cite{apxdomin}.

The relationship between the two parameters becomes clearer when we notice
that their fractional relaxations coincide.
\begin{itemize}
\item The \emph{fractional domination number} $\gamma^\star(G)$ of a graph $G$ is the minimum of $\sum_{v\in V(G)} f(v)$ over all functions
$f:V(G)\to\mathbb{R}_{\ge 0}$ such that $$\sum_{v\in N[u]} f(v)\ge 1$$
holds for every vertex $u\in V(G)$.  These constraints are satisfied by the characteristic function
of a dominating set, and thus $\gamma^\star(G)\le\gamma(G)$.
\item The \emph{fractional $2$-independence number} $\alpha_2^\star(G)$ of $G$
is the maximum of $\sum_{v\in V(G)} g(v)$ over all functions
$g:V(G)\to\mathbb{R}_{\ge 0}$ such that $$\sum_{v\in N[u]} g(v)\le 1$$
holds for every vertex $u\in V(G)$.  These constraints are satisfied by the characteristic function
of a $2$-independent set, and thus $\alpha_2(G)\le \alpha^\star_2(G)$.
\end{itemize}
It is easy to see that the linear programs defining $\gamma^\star(G)$
and $\alpha^\star_2(G)$ are dual to each other, and consequently $\alpha^\star_2(G)=\gamma^\star(G)$.

Thus, it is natural to ask for which graphs $G$ the gap between
$\alpha_2(G)$ and $\gamma(G)$ is small. For further motivation, note that both domination number and $2$-independence number
of a graph are hard to approximate in general~\cite{domcompl,khot2001improved}\footnote{For any graph $G$, if the graph $G'$ is obtained from $G$ by subdividing every edge exactly once and adding a vertex
adjacent to all the subdivision vertices, then $\alpha(G)\le \alpha_2(G')\le \alpha(G)+1$, and thus the hardness of approximation
of the $2$-independence number follows from the hardness of approximation of the independence number.}.
However, if the ratio $\gamma(G)/\alpha_2(G)$ is small, then they are both approximated by the value of their common fractional relaxation, which can be determined in polynomial time.
Thus, we define the \emph{domination-to-2-independence ratio} of a graph class~$\GG$
as $\sup_{G\in\GG} \tfrac{\gamma(G)}{\alpha_2(G)}$.
\begin{problem}\label{prob-linear}
Which graph classes have bounded domination-to-2-independence ratio?
\end{problem}
Let us remark that the following relaxed version of this problem also seems non-trivial.
We say that a graph class~$\GG$ has \emph{domination tied to 2-independence} if there exists a function $h$ such that $\gamma(G)\le h(\alpha_2(G))$
holds for every graph $G\in \GG$.  Note that the class~$\GG$ has bounded domination-to-2-independence ratio exactly when this
function $h$ is linear.
\begin{problem}\label{prob-tied}
Which graph classes have domination tied to 2-independence?
\end{problem}
Problem~\ref{prob-linear} has been rather thoroughly studied, often from the perspective of the following equivalent interpretation:
The domination number corresponds to covering all vertices of the graph by the minimum number of closed neighborhoods,
while the $2$-independence number can be seen as packing maximum number of pairwise-disjoint closed neighborhoods in the graph.
The following graph classes are known to have bounded domination-to-2-independence ratio:
\begin{itemize}
\item Trees~\cite{trees}, strongly chordal graphs~\cite{stroch}, and dually chordal graphs~\cite{duach} (ratio $1$).
\item Outerplanar graphs~\cite{bonamy2025graph}, connected biconvex graphs~\cite{gg}, and claw-free subcubic graphs~\cite{delta} (ratio at most $2$).
\item Bipartite cubic graphs~\cite{gg} (ratio at most $\tfrac{120}{49}$).
\item Asteroidal triple-free graphs~\cite{bonamy2025graph} and convex graphs~\cite{bonamy2025graph} (ratio at most $3$).
\item $2$-degenerate graphs~\cite{bonamy2025graph} (ratio at most $7$).
\item Unit disk graphs~\cite{bonamy2025graph} (ratio at most $32$).
\item Graphs of maximum degree $\Delta$~\cite{delta} (ratio at most $\Delta$).
\item $K_{q,r}$-minor-free graphs~\cite{bohme2003domination} (ratio less than $4r+(q-1)(r+1)$),
and consequently all proper minor-closed classes.
\item Better bounds are known for specific proper minor-closed classes~\cite{bonamy2025graph}: planar graphs (ratio at most $10$) and
graphs of treewidth $k$ (ratio at most $k$).
\item Graphs of twinwidth at most $k$~\cite{bonamy2025graph} (ratio at most $4k^2$).
\item Graphs with weak 2-coloring number at most $c$~\cite{apxdomin} (ratio at most $c^2$).
\end{itemize}
Let us remark that all graph classes with bounded expansion (which includes many natural classes of sparse graphs,
and in particular all proper topological minor-closed graph classes) have bounded weak $2$-coloring numbers~\cite{zhu2009colouring};
thus, the last result is a qualitative generalization of many of the previous ones (though giving worse bounds
on the domination-to-2-independence ratio).
For the negative results, the following graph classes do not have domination tied to $2$-independence:
\begin{itemize}
\item Cartesian products of cliques~\cite{caprod}.
\item Bipartite $3$-degenerate graphs~\cite{apxdomin}.
\item Split graphs~\cite{bonamy2025graph}.
\end{itemize}

Many of these previously considered graph classes are \emph{monotone} (closed on subgraphs) and have bounded average degree.
For these graph classes, the ratio between the domination number and the fractional domination number is bounded by a
constant~\cite{BANSAL201721} (in fact, an even stronger claim holds, see Lemma~\ref{lemma-near2ind} below).  Thus, in this setting, the problem reduces to considering the
relationship between the $2$-independence number and its fractional relaxation.
As our main result, we solve Problem~\ref{prob-linear} for these graph classes.
That is, we provide an exact characterization of monotone graph classes with bounded average degree
that have bounded domination-to-2-independence ratio.

Specifically, we show that the following structure is the only obstruction.
A \emph{cascade} is a pair $(H,\psi)$, where $\psi$ is a proper
coloring of a graph $H$ by colors \texttt{v}, \texttt{e}, and \texttt{d}
such that every vertex of color \texttt{e} has exactly two neighbors
of color \texttt{v} and exactly one neighbor of color \texttt{d},
and $H$ has no other edges. We say that the cascade \emph{appears} in a graph $G$ if $H$ is a subgraph of $G$,
and that a graph class~$\GG$ \emph{contains} the cascade if $H\in\GG$.
The \emph{foundation} of the cascade is the graph $F$ with vertex
set $\psi^{-1}(\mathtt{v})$ and with two vertices of $F$ adjacent
if and only if they have a common neighbor (of color \texttt{e}) in $H$.
Thus, the graph~$H$ is obtained from the \emph{$1$-subdivision} of $F$ (the graph resulting from subdividing each edge of $F$ exactly once) by adding vertices (of color \texttt{d})
dominating the subdivision vertices.
The \emph{slope} of the cascade is defined as
$$\frac{|\psi^{-1}(\mathtt{v})|}{\max(\alpha(F),|\psi^{-1}(\mathtt{d})|)};$$
i.e., the cascade has slope at least $s$ if and only if
\begin{align*}
\alpha(F)&\le \frac{|\psi^{-1}(\mathtt{v})|}{s}\text{ and}\\
|\psi^{-1}(\mathtt{d})|&\le \frac{|\psi^{-1}(\mathtt{v})|}{s}.
\end{align*}

\begin{observation}\label{obs-casc}
If $(H,\psi)$ is a cascade of slope $s$, then $\gamma(H)\ge \tfrac{s}{4}\alpha_2(H)$.
\end{observation}
\begin{proof}
Note that the closed neighborhood of any vertex of $H$ contains
at most two vertices of $\psi^{-1}(\mathtt{v})$, and thus
\begin{equation}\label{ineq-cascgamma}
\gamma(H)\ge \frac{1}{2}|\psi^{-1}(\mathtt{v})|.
\end{equation}
On the other hand, consider any $2$-independent set $A$ in $H$.
Observe that $A\cap \psi^{-1}(\mathtt{v})$ is an independent set
in the foundation $F$ of the cascade $(H,\psi)$, and thus
$$|A\cap \psi^{-1}(\mathtt{v})|\le \alpha(F)\le \frac{|\psi^{-1}(\mathtt{v})|}{s}.$$
Moreover, all vertices of $V(H)\setminus \psi^{-1}(\mathtt{v})$ are contained
in closed neighborhoods of vertices of color \texttt{d}, and thus
$$|A\setminus \psi^{-1}(\mathtt{v})|\le |\psi^{-1}(\mathtt{d})|\le \frac{|\psi^{-1}(\mathtt{v})|}{s}.$$
Consequently,
$$\alpha_2(H)\le \frac{2}{s}|\psi^{-1}(\mathtt{v})|.$$
The claim of the lemma follows by combining this inequality with (\ref{ineq-cascgamma}).
\end{proof}
Thus, if a class of graphs contains cascades of arbitrarily large slope, then it does not
have bounded domination-to-2-independence ratio.  We prove that the converse holds for monotone graph classes
of bounded average degree.
\begin{theorem}\label{thm-main}
The following claims are equivalent for any monotone graph class~$\GG$ with bounded average degree:
\begin{itemize}
\item The class~$\GG$ has bounded domination-to-2-independence ratio.
\item There exists a real number $c$ such that every cascade contained in $\GG$ has slope at most $c$.
\end{itemize}
\end{theorem}

Of course, Theorem~\ref{thm-main} can be used to reprove (in the qualitative sense, i.e., without providing
exact upper bounds) all previous results concerning
boundedness of the domination-to-2-independence ratio in specific monotone classes with bounded average degree.
For this, it is useful to know that the set of vertices of color \texttt{e} in a cascade of large slope must be large.
This easily follows from a standard consequence of Turán's theorem.
\begin{lemma}\label{lemma-alpha}
The average degree of every graph $F$ is at least $\tfrac{|V(F)|}{\alpha(F)}-1$. Equivalently,
if $F$ has average degree at most $t$, then $\alpha(F)\ge\tfrac{|V(F)|}{t+1}$.
\end{lemma}

This lemma applied to the foundation of a cascade gives the following bound.

\begin{corollary}\label{cor-elarge}
The foundation of a cascade $(H,\psi)$ of slope $s\ge 1$ has average degree at least $s-1$, and consequently
$$|\psi^{-1}(\mathtt{e})| \ge \frac{s-1}{2}|\psi^{-1}(\mathtt{v})|\ge \frac{s-1}{4}(|\psi^{-1}(\mathtt{v})|+s|\psi^{-1}(\mathtt{d})|)\ge \frac{s-1}{4}|V(H)\setminus \psi^{-1}(\mathtt{e})|.$$
\end{corollary}

For a cascade $(H,\psi)$, recall that $H-\psi^{-1}(\mathtt{d})$ is the $1$-subdivision of the foundation of the cascade.
The $1$-subdivision of a graph $F$ of average degree $t$ has weak $2$-coloring number $\Omega(t)$, see~\cite{subdivchar},
and thus every cascade appearing in a graph of weak $2$-coloring number at most $c$ has slope $O(c)$.
Moreover, the weak $2$-coloring number of a graph is an upper bound on the average degree of all subgraphs of this graph,
and consequently Theorem~\ref{thm-main} implies a weaker form of the main result of~\cite{apxdomin}.
\begin{corollary}\label{cor-arrang}
Any class of graphs with bounded weak $2$-coloring number has bounded domination-to-2-independence ratio.
\end{corollary}

Moreover, Corollary~\ref{cor-elarge} implies that the average degree of cascades of large slope is close to $6$,
which has the following consequence (the \emph{maximum average degree} of a graph is the maximum of the average degrees of its subgraphs).
\begin{corollary}\label{cor-avg}
The class~$\GG_d$ of graphs of maximum average degree at most $d$ has bounded domination-to-2-independence ratio when $d<6$,
and does not have domination tied to 2-independence when $d\ge 6$.
\end{corollary}
\begin{proof}
By Corollary~\ref{cor-elarge}, if a cascade $(H,\psi)$ has slope $s>1$, then $H$ has average degree
$$\frac{2|E(H)|}{|V(H)|}=\frac{6|\psi^{-1}(\mathtt{e})|}{|\psi^{-1}(\mathtt{e})|+|V(H)\setminus \psi^{-1}(\mathtt{e})|}
\ge \frac{6|\psi^{-1}(\mathtt{e})|}{|\psi^{-1}(\mathtt{e})|+\frac{4}{s-1}|\psi^{-1}(\mathtt{e})|}=6\cdot \frac{s-1}{s+3}.$$
Thus, if $d<6$ and $H\in \GG_d$, then $s\le \tfrac{24}{6-d}-3$, and thus $\GG_d$ has bounded domination-to-2-independence ratio
by Theorem~\ref{thm-main}.

On the other hand, suppose that $d\ge 6$.
For any integer $n\ge 4$, consider the graph $K'_n$ obtained from the $1$-subdivision of the clique $K_n$ by adding a vertex
adjacent to all vertices of degree two (by coloring the vertices appropriately, the graph~$K'_n$ can be turned into a cascade
with foundation $K_n$ and exactly one vertex of color~\texttt{d}).  Note that the graph $K'_n$ is $3$-degenerate, and thus
its maximum average degree is less than~6.  Consequently, $\{K'_n:n\ge 4\}\subset \GG_d$. 
Moreover, we have $\gamma(K'_n)\ge n/2$ and $\alpha_2(K'_n)=2$,
and thus there cannot exist any function $h$ such that $\gamma(G)\le h(\alpha_2(G))$ for every $G\in\GG_d$.
\end{proof}
Note that $2$-degenerate graphs have maximum average degree less than $4$,
and thus Corollary~\ref{cor-avg} strengthens one of the main results of~\cite{bonamy2025graph}.

Our work does not resolve Problem~\ref{prob-tied}, since there are sparse monotone graph classes
that have domination tied to 2-independence, but unbounded domination-to-2-independence ratio.
As an example, for any integer $k\ge 3$ and positive real number $d$, let $\CC_{k,d}$ be the class of graphs of maximum average degree at most $d$
that do not contain the $1$-subdivision of $K_k$ as a subgraph.
\begin{theorem}\label{thm-vc2}
For any integer $k\ge 3$ and any positive real number $d$, the class~$\CC_{k,d}$ has domination tied to 2-independence.
\end{theorem}
It is easy to see that every cascade with triangle-free foundation is contained in $\CC_{k,d}$
when $k\ge 5$ and $d\ge 6$.  Since such cascades can have arbitrarily large slope,
the class $\CC_{k,d}$ does not have bounded domination-to-2-independence ratio.
Thus, Problem~\ref{prob-tied} for sparse monotone graph classes remains as a natural open question.

Can our results be strengthened?  It would be more satisfying to obtain the answer for hereditary graph classes
(i.e., classes closed on induced subgraphs) rather than monotone ones, but in our argument, the subgraphs that
give rise to cascades can contain edges between vertices of colors \texttt{v} and \texttt{d} and we were not
able to deal with this issue.

The restriction to graphs of bounded maximum average degree is quite important
for us: As mentioned above, the domination number of any such graph is most linear in its fractional domination number,
which allows us to consider only the relationship between the $2$-independence number and its fractional relaxation.
As a consequence, our understanding of domination-to-2-independence ratio of dense graph classes remains rather limited.

Finally, let us remark that even if a cascade of large slope appears in a graph~$G$ as an induced subgraph, this does not imply that
the ratio $\gamma(G)/\alpha_2(G)$ is large; e.g., the graph~$G$ could contain a universal vertex outside of this cascade, forcing $\gamma(G)=\alpha_2(G)=1$.
In other words, we obtain a characterization for graph classes, not for individual graphs.   A strong dichotomy
of form ``either $\gamma(G)/\alpha_2(G)$ is small, or the graph $G$ contains a nice substructure forcing $\gamma(G)/\alpha_2(G)$ to be large''
would of course be preferable, but seems much harder to obtain even with substantial additional restrictions.

The rest of the paper is devoted to the proof of our main result, Theorem~\ref{thm-main} (in Section~\ref{sec-main}),
and to the proof of Theorem~\ref{thm-vc2} (in Section~\ref{sec-vc2}).  Before that, we need to gather a few auxiliary results.

\section{Preliminaries}

For the proof of Theorem~\ref{thm-main}, it will be convenient to work with a relaxed version of cascades.
For an integer~$m\ge 2$, an \emph{$m$-cascade} is a pair $(H,\psi)$, where $\psi$ is a proper
coloring of a graph $H$ by colors \texttt{v}, \texttt{e}, and \texttt{d}
such that every vertex of color \texttt{e} has at most $m$ neighbors
of color \texttt{v} and exactly one neighbor of color \texttt{d},
and $H$ has no other edges.   We define the foundation of an $m$-cascade in
the same way as for cascades.  For real numbers $s_1,s_2\ge 1$, we say that
an $m$-cascade $(H,\psi)$ with foundation $F$ has \emph{slope at least $(s_1,s_2)$}
if
$$\alpha(F)\le \tfrac{|\psi^{-1}(\mathtt{v})|}{s_1}\text{ and }|\psi^{-1}(\mathtt{d})|\le \tfrac{|\psi^{-1}(\mathtt{v})|}{s_2}.$$
Let us now argue that these relaxed cascades can be turned into normal ones;
basically, we show that simply choosing two edges from each vertex of color $\mathtt{e}$
to $\psi^{-1}(\mathtt{v})$ at random results with non-zero probability in a cascade of large slope.

To formulate the argument precisely, let us introduce the following definitions.
Let $E$ and $V$ be finite sets and let $S=\{S_e:e\in E\}$ be a system of (not necessarily pairwise different) subsets
of $V$.  We let $G_{V,E}$ be the graph with vertex set $V$ and with distinct vertices $u,v\in V$ adjacent if
and only if there exists $e\in E$ with $\{u,v\}\in S_e$.  We say that a graph $G'\subseteq G_{V,E}$
with vertex set $V$ is \emph{compatible} with the system $S$ if there exists an injective function $f:E(G')\to E$
such that $\{u,v\}\subseteq S_{f(uv)}$ holds for every edge $uv\in E(G')$.

\begin{lemma}\label{lemma-mksim}
Let $m\ge 2$ be an integer, let $E$ and $V$ be finite sets and let $n=|V|$.
Let $S=\{S_e:e\in E\}$ be a system of (not necessarily pairwise different) subsets
of~$V$, each of size at most~$m$.  Let $c\ge 1$ be a real number and let $b=2\binom{m}{2}^2\log(ec)+2$.
If $\alpha(G_{V,E})\le \tfrac{n}{bc}$, then there exists a graph $G'\subseteq G_{V,E}$ compatible with $S$
such that $\alpha(G')\le \tfrac{n}{c}$.
\end{lemma}
\begin{proof}
Let $G'$ be the random graph obtained as follows.
For each $e\in E$ such that $|S_e|\ge 2$, we choose a subset $\{u,v\}\subseteq S_e$ of size two independently uniformly at random
and join $u$ with $v$ by an edge.  Clearly, this ensures that $G'$ is compatible with $S$.

Let us now consider any set $A\subseteq V$ of size $a=\lceil n/c\rceil$.
Since
$$\alpha(G_{V,E}[A])\le \alpha(G_{V,E})\le \frac{n}{bc}\le \frac{a}{b},$$
Lemma~\ref{lemma-alpha} implies that the graph $G_{V,E}[A]$ has average degree at least $b-1$,
and thus $G_{V,A}[A]$ has at least $\tfrac{b-1}{2}a$ edges.  Consequently,
there exist at least
$$\frac{\tfrac{b-1}{2}a}{\binom{m}{2}}>\binom{m}{2}\log(ec)\cdot a$$ elements $e\in E$
such that $|S_e\cap A|\ge 2$.  Each such element $e$ independently contributes an edge
of $G'$ with both ends in $A$ with probability at least $\binom{m}{2}^{-1}$.
Thus, the probability that $A$ is an independent set in $G'$ is less than
$$\Bigl(1-\binom{m}{2}^{-1}\Bigr)^{\binom{m}{2}\log(ec)\cdot a}\le e^{-\log(ec)\cdot a}=(ec)^{-a}.$$
The number of subsets of $V$ of size $a$ is at most
$$\binom{n}{a}\le \Bigl(\frac{en}{\lceil n/c\rceil}\Bigr)^a\le (ec)^a.$$
Consequently, the expected number of independent sets of size $a$ in $G'$ is less than $1$,
and thus $\alpha(G')\le a-1\le n/c$ holds with non-zero probability.
\end{proof}

By applying this lemma to the system of neighborhoods of vertices of color \texttt{e}, we obtain the
desired result on turning relaxed cascades into normal ones.

\begin{corollary}\label{cor-relaxcasc}
For every real number $c\ge 1$ and integer $m\ge 2$, if an $m$-cascade $(H,\psi)$ has slope at least $(m^4c\log(ec),c)$,
then a cascade of slope at least $c$ appears in $H$.
\end{corollary}
\begin{proof}
Let $F$ be the foundation of $(H,\psi)$; then $\alpha(F)\le \tfrac{|\psi^{-1}(\mathtt{v})|}{m^4c\log(ec)}$.
Note that $2\binom{m}{2}^2\log(ec)+2\le m^4\log(ec)$.
Hence, Lemma~\ref{lemma-mksim} implies that there exists a spanning subgraph $F_1$ of $F$
with $\alpha(F_1)\le \tfrac{|\psi^{-1}(\mathtt{v})|}{c}$ and an injective function $f:E(F_1)\to \psi^{-1}(\mathtt{e})$
such that for every $uv\in E(F_1)$, the vertex $f(uv)$ is a common neighbor of $u$ and $v$ in $H$.

Let $H_1$ be the subgraph of $H$ obtained by deleting all vertices of color \texttt{e} not in the image of $f$
and all edges between vertices of colors \texttt{v} and \texttt{e} except for those in $\{uf(uv), vf(uv):uv\in E(F_1)\}$.  Let $\psi_1$ be the restriction of $\psi$ to $V(H_1)$.  Then $(H_1,\psi_1)$ is a cascade
with foundation $F_1$.  Moreover, since $\alpha(F_1)\le \tfrac{|\psi^{-1}(\mathtt{v})|}{c}=\tfrac{|\psi_1^{-1}(\mathtt{v})|}{c}$ and
$|\psi^{-1}_1(\mathtt{d})|=|\psi^{-1}(\mathtt{d})|\le \tfrac{|\psi^{-1}(\mathtt{v})|}{c}=\tfrac{|\psi_1^{-1}(\mathtt{v})|}{c}$,
the slope of the cascade $(H_1,\psi_1)$ is at least~$c$.
\end{proof}

We also need the following technical lemma that shows that the removal of a small number of edges
cannot substantially increase the independence number.
\begin{lemma}\label{lemma-noincr}
Let $F$ be a graph and let $a=\tfrac{\alpha(F)}{|V(F)|}$. Let $F'$ be a subgraph of $F$ and let $r\ge 1$ be a real number.
If $|E(F)\setminus E(F')|\le r|V(F)|$, then
$$\alpha(F')\le 2\sqrt{ra}\cdot |V(F)|.$$
\end{lemma}
\begin{proof}
Suppose for a contradiction that $F'$ contains an independent set $C$ of size more than $2\sqrt{ra}\cdot|V(F)|$.
All edges of the induced subgraph $F[C]$ belong to $E(F)\setminus E(F')$, and thus there are at most $r|V(F)|$ of them.
Consequently, $F[C]$ has average degree is at most $\tfrac{2r|V(F)|}{|C|}<\sqrt{r/a}$.  
Note that $a\le 1$, and thus $\sqrt{r/a}\ge 1$.
By Lemma~\ref{lemma-alpha}, we have
$$\alpha(F)\ge \alpha(F[C])\ge \frac{|C|}{\sqrt{r/a}+1}\ge \frac{|C|}{2\sqrt{r/a}}>a|V(F)|.$$
Since $\alpha(F)=a|V(F)|$, this is a contradiction.
\end{proof}

\section{Graphs without steep cascades}\label{sec-main}

Hall's theorem has the following well-known consequence.

\begin{lemma}\label{lemma-outdeg}
For every positive integer $d$, every graph of maximum average degree at most $2d$ has an orientation of maximum outdegree
at most $d$.
\end{lemma}

A set $S$ of vertices of a graph $G$ is \emph{$k$-nearly $2$-independent} if the closed neighborhood of every vertex of $G$
intersects $S$ in at most $k$ vertices.  Dvořák~\cite{apxlp} proved the following.

\begin{lemma}\label{lemma-near2ind}
For every positive integer $d$, every graph $G$ with an orientation of maximum outdegree at most $d$ contains
a $(2d+1)$-nearly $2$-independent set of size at least $\tfrac{1}{4d+2}\gamma(G)$.
\end{lemma}

Thus, for graphs of bounded maximum average degree, the problem of relating $\alpha_2(G)$ to $\gamma(G)$
reduces to the question of whether (for a suitably chosen $m$) a large $m$-nearly $2$-independent set $B$ can be turned into a $2$-independent one.
We need to show that if this is not possible, then an $m$-cascade of large slope appears in $G$.
Our plan is to use a large part of $B$ for the vertices of color \texttt{v} of this $m$-cascade,
and most of the rest of the vertices for those of color \texttt{e}.  For this to succeed, we need to argue
that almost all vertices of $V(G)\setminus B$ can be dominated by a set of much less than $|B|$ vertices
(which then will be used as vertices of color \texttt{d}).  This indeed is the case, as we show next.
\begin{lemma}\label{lemma-outdom}
Let $d$ and $m$ be positive integers such that $d<m$, let $G$ be a graph of maximum average degree
at most $2d$, and let $b$ be the maximum size of an $m$-nearly $2$-independent set in $G$.
For every set $B\subseteq V(G)$, there exists sets $X\subseteq V(G)\setminus B$ and $D\subseteq V(G)$ such that
$|X|\le b$, $|D|\le \tfrac{(d+1)b}{m-d}$, and every vertex of $V(G)\setminus (B\cup X\cup D)$ has a neighbor in $D$.
\end{lemma}
\begin{proof}
By Lemma~\ref{lemma-outdeg}, we can fix an orientation $\vec{G}$ of $G$ of maximum outdegree at most $d$.
We say that a set $A$ of vertices is \emph{out-2-independent} if it is independent in $G$
and distinct vertices of $A$ do not have any common out-neighbors in $\vec{G}$; i.e., every vertex of $\vec{G}$ has at most one in-neighbor in $A$.

Let us choose pairwise disjoint set $X_1,X_2,\ldots, X_{m-d}\subset V(G)\setminus B$
so that for $i=1,\ldots, m-d$, the set $X_i$ is an inclusionwise-maximal subset of $V(G)\setminus (B\cup X_1\cup \ldots \cup X_{i-1})$
which is out-2-independent in $\vec{G}$.  Let $X=X_1\cup \ldots \cup X_{m-d}$.  Since the sets $X_1$, \ldots, $X_{m-d}$ are
out-2-independent in $\vec{G}$, every vertex of $V(G)\setminus X$ has at most $m-d$ in-neighbors in $X$
and every vertex of $X$ has at most $m-d-1$ in-neighbors in $X$.
Since $\vec{G}$ has maximum outdegree at most $d$, it follows that the set $X$ is $m$-nearly $2$-independent in $G$,
and thus $|X|\le b$.  Therefore, we can fix an index $i\in \{1,\ldots,m-d\}$ such that $|X_i|\le \tfrac{b}{m-d}$.

Let $D$ be the set consisting of $X_i$ and of all out-neighbors of vertices of $X_i$ in~$\vec{G}$;
then $|D|\le (d+1)|X_i|\le \tfrac{(d+1)b}{m-d}$.
We claim that every vertex $v\in V(G)\setminus (B\cup X\cup D)$ has an out-neighbor in $D$.
Indeed, since $X_i$ is an inclusionwise-maximal out-2-independent subset of
$V(G)\setminus (B\cup X_1\cup \ldots \cup X_{i-1})\supseteq V(G)\setminus (B\cup X\cup D)$,
the set $X_i\cup\{v\}$ is not out-2-independent in $\vec{G}$.
If $v$ has a neighbor $u\in X_i$, then the edge $uv$ must be directed towards $u$,
as otherwise we would have $v\in D$ (and note that $u\in D$, since $X_i\subseteq D$).  If $v$ does not have any neighbors in $X_i$, then $v$ and a vertex of $X_i$ must
have a common out-neighbor, and this out-neighbor belongs to $D$.
\end{proof}

We are now ready to execute the key step towards the proof of Theorem~\ref{thm-main} that we outlined
above, i.e., using a large $m$-nearly $2$-independent set to obtain an $m$-cascade of large slope.
\begin{lemma}\label{lemma-2ind-or-steep}
Let $d$, $s$, $c$, and $m\ge (2c+1)d+2c$ be positive integers.
Let $G$ be a graph of maximum average degree at most $2d$ and let $b$ be the maximum
size of an $m$-nearly $2$-independent set in $G$.  If $\alpha_2(G)<\tfrac{b}{48m^2s^2}$, then an $m$-cascade
of slope at least $(s,c)$ appears in $G$.
\end{lemma}
\begin{proof}
Let $B$ be an $m$-nearly $2$-independent set in $G$ of size $b$.
We plan to use a large part of $B$ for the color-\texttt{v} vertices of the $m$-cascade.
For that, we need to argue that it is possible to select vertices of color \texttt{e} so that
the corresponding foundation graph has small independence number.

Let us first consider the auxiliary graph $F$ with vertex set $B$ and with distinct vertices $u,v\in B$ adjacent
exactly if they have a common neighbor in $G$ (possibly even one belonging to $B$).
We claim that $\alpha(F)\le 2\alpha_2(G)$.  Indeed, consider any independent set $A$ in $F$.
No two vertices in $A$ have a common neighbor in $G$, and in particular the subgraph $G[A]$ has maximum degree at most one.
Thus, there exists a set $A_0\subseteq A$ of size at least $|A|/2$ which is also independent in~$G$.
Since no two vertices of $A_0$ have a common neighbor in $G$, the set $A_0$ is $2$-independent in $G$, and thus
$|A|\le 2|A_0|\le 2\alpha_2(G)$.  By the assumptions of the Lemma, it follows that
\begin{equation}\label{eq-smis}
\alpha(F)<\frac{b}{24m^2s^2}.
\end{equation}
By Lemma~\ref{lemma-outdom}, there exist sets $X\subseteq V(G)\setminus B$ and $D\subseteq V(G)$ such that
$|X|\le b$, $|D|\le \tfrac{(d+1)b}{m-d}$, and every vertex of $V(G)\setminus (B\cup X\cup D)$ has a neighbor in~$D$.
Since $m\ge (2c+1)d+2c$, we have
\begin{equation}\label{eq-di}
|D|\le \frac{b}{2c}\le\frac{b}{2}.
\end{equation}
We are going to use $D$ for the color-\texttt{d} vertices of the $m$-cascade.
A minor issue with this is that some vertices of $B\cup D$ (or vertices of $X$, which are not necessarily dominated by $D$)
can be contributing edges to the candidate $F$ for the foundation of this $m$-cascade.
Thus, let $F'$ be the spanning subgraph of $F$ where distinct vertices $u$ and $v$ are adjacent if and only if
in $G$, they have a common neighbor belonging to $V(G)\setminus (B\cup X\cup D)$.
Since $B=V(F')$ is an $m$-nearly $2$-independent set in $G$, we have
$$|E(F)\setminus E(F')|\le \binom{m}{2}|B\cup X\cup D|<\frac{3m^2b}{2}.$$
Let $r=\tfrac{3}{2}m^2$ and $a=\tfrac{\alpha(F)}{|V(F)|}=\tfrac{\alpha(F)}{b}$; by (\ref{eq-smis}), we have $a<\tfrac{1}{24m^2s^2}$.
By Lemma~\ref{lemma-noincr}, we have
$$\alpha(F')\le 2\sqrt{ra}\cdot b<2\sqrt{\frac{3m^2}{2\cdot 24m^2s^2}}\cdot b=\frac{b}{2s}.$$
Let $S_{\mathtt{v}}=B\setminus D$; by (\ref{eq-di}), we have $|S_{\mathtt{v}}|\ge b/2$, and thus
$$\alpha(F'[S_{\mathtt{v}}])\le \alpha(F')\le \frac{b}{2s}\le \frac{|S_{\mathtt{v}}|}{s}.$$
Let $S_{\mathtt{d}}=D$; by (\ref{eq-di}), we have
$$|S_{\mathtt{d}}|\le \frac{b}{2c}\le \frac{|S_{\mathtt{v}}|}{c}.$$
Let $S_{\mathtt{e}}=V(G)\setminus (B\cup X\cup D)$.
Let $H$ be the subgraph of $G$ with vertex set $S_{\mathtt{v}}\cup S_{\mathtt{e}}\cup S_{\mathtt{d}}$
containing for each vertex $v\in S_{\mathtt{e}}$
\begin{itemize}
\item a single edge from $v$ to a neighbor of $v$ in $G$ belonging to $S_{\mathtt{d}}=D$ and
\item all edges from $v$ to neighbors of $v$ in $G$ belonging to $S_{\mathtt{v}}$
(there are at most $m$ of them).
\end{itemize}
Let $\psi$ be
the coloring of $H$ assigning the color \texttt{v} to all vertices of $S_{\mathtt{v}}$, the color~\texttt{e} to all vertices of $S_{\mathtt{e}}$,
and the color \texttt{d} to all vertices of $S_{\mathtt{d}}$.  Then $(H,\psi)$ is an $m$-cascade with foundation $F'[S_{\mathtt{v}}]$
and slope at least $(s,c)$ appearing in $G$.
\end{proof}

Our main result now easily follows.

\begin{proof}[Proof of Theorem~\ref{thm-main}]
If the class~$\GG$ contains cascades of arbitrarily large slope, then by Observation~\ref{obs-casc},
it does not have bounded domination-to-2-independence ratio.

Suppose now that there exists a positive integer $c$ such that every cascade contained in $\GG$ has slope less than $c$.
Let $d\ge 3$ be an integer such that every graph in $\GG$ has average degree at most $2d$; since the class~$\GG$ is monotone,
all graphs in $\GG$ have maximum average degree at most $2d$.  Let $m=(2c+1)d+2c$ and $s=\lceil m^4c\log(ec)\rceil$.
By Corollary~\ref{cor-relaxcasc}, no $m$-cascade of slope at least $(s,c)$ is contained in $\GG$.
We claim that
$$\gamma(G)\le 48(4d+2)m^2s^2\cdot \alpha_2(G)$$
holds for every $G\in \GG$.  Indeed, by Lemmas~\ref{lemma-outdeg} and \ref{lemma-near2ind}, $G$ contains a $(2d+1)$-nearly $2$-independent
set of size at least $\tfrac{1}{4d+2}\gamma(G)$.  Let $b$ be the maximum size of an $m$-nearly $2$-independent set in $G$;
since $m>2d+1$, we have $b\ge \tfrac{1}{4d+2}\gamma(G)$.  Since $\GG$ is monotone and no $m$-cascade of slope at least $(s,c)$ is contained in $\GG$,
no $m$-cascade of slope at least $(s,c)$ appears in $G$.  Therefore, by Lemma~\ref{lemma-2ind-or-steep},
we have
$$\alpha_2(G)\ge \frac{b}{48m^2s^2}\ge \frac{\gamma(G)}{48(4d+2)m^2s^2}.$$
It follows that the class~$\GG$ has bounded domination-to-2-independence ratio.
\end{proof}

\section{Graphs without clique subdivisions}\label{sec-vc2}

Finally, we show that the class $\CC_{k,d}$ has domination tied to $2$-independence.

\begin{proof}[Proof of Theorem~\ref{thm-vc2}]
Without loss of generality, we can assume that $d\ge 2$.
Let $d'=\lceil d/2\rceil$, so that $d\le 2d'\le d+2$.
For a positive integer $m$, let $I(m)$ be the multicolor Ramsey number $R\bigl(m,\binom{2d'+1}{2}\times k\bigr)$;
that is, for every (not necessarily proper) coloring of the edges of a clique with at least $I(m)$ vertices using
$\binom{2d'+1}{2}+1$ colors, there exists either a subclique of size $m$ whose edges all have the first color,
or a subclique of size $k$ whose edges all have one of the remaining colors.
For every positive integer $x$, let $h(x)=8d^2I(x+1)$; we claim that every graph $G\in\CC_{k,d}$ satisfies $\gamma(G)\le h(\alpha_2(G))$.

Indeed, consider any such graph $G$.  By Lemma~\ref{lemma-outdeg}, $G$ has an orientation of maximum outdegree
at most $d'$, and by Lemma~\ref{lemma-near2ind}, $G$ contains a $(2d'+1)$-nearly $2$-independent set $A'$ of size at least
$\tfrac{1}{4d'+2}\gamma(G)\ge \tfrac{1}{2d+6}\gamma(G)$.   By Lemma~\ref{lemma-alpha} applied to the graph $G[A']$ (whose average degree is at most $d$), there exists
an independent set $A\subseteq A'$ of size at least $\tfrac{|A'|}{d+1}\ge \tfrac{1}{8d^2}\gamma(G)$.

Let $\varphi$ be a (not necessarily proper) coloring of the edges of the clique $K$
with vertex set $A$ using colors $\bigl\{0,\ldots,\binom{2d'+1}{2}\bigr\}$ obtained as follows.
Each edge $xy\in E(K)$ such that the vertices $x$ and $y$ have a common neighbor in $G$ (necessarily outside
of $A$, since $A$ is an independent set) is assigned to an arbitrarily chosen common neighbor.
Since the set $A$ is $(2d'+1)$-nearly $2$-independent, at most $\binom{2d'+1}{2}$ edges of $K$ are assigned to each vertex in $V(G)\setminus A$.
For each vertex in $V(G)\setminus A$, we give each edge of $K$ assigned to it a distinct color in $\bigl\{1,\ldots,\binom{2d'+1}{2}\bigr\}$.
The edges $xy$ of $K$ such that the vertices $x$ and $y$ do not have any common neighbors in $G$ are given color $0$.

Note that for every $i\in \bigl\{1,\ldots,\binom{2d'+1}{2}\bigr\}$, the $1$-subdivision of the subgraph of $K$
formed by the edges of color $i$ is a subgraph of $G$: Indeed, the ends of any such edge have a common neighbor in $G$,
and these common neighbors are pairwise different for distinct edges of the same color by the construction of $\varphi$.
Since $G$ does not contain the $1$-subdivision of $K_k$ as a subgraph, it follows that $K$ does not contain any subclique
of size $k$ with all edges of color $i$.

Let $m$ be the largest integer such that $I(m)\le \tfrac{1}{8d^2}\gamma(G)$.
By the definition of $I(m)$, we conclude that $K$ contains a subclique $Q$ of size $m$ with all edges of color $0$.
By the definition of the coloring of $K$, the vertices of $Q$ form a $2$-independent set in $G$.
Therefore,
$$h(\alpha_2(G))\ge h(|V(Q)|)=h(m)=8d^2I(m+1)>\gamma(G)$$
by the choice of $m$.
\end{proof}

\section*{Acknowledgments}

We would like to thank Fanny Hauser for insightful discussions.

\bibliographystyle{alpha}
\bibliography{bibliography}

\end{document}